\documentclass[12pt]{article}

\usepackage{xr}
\usepackage[english]{babel}
\usepackage{lmodern}
\usepackage[dvipsnames]{xcolor}
\usepackage{array}
\usepackage{amsmath}
\usepackage{amssymb}
\usepackage{amsfonts}
\usepackage{amsthm}
\usepackage{enumitem}
\usepackage{mathabx}
\usepackage{mathrsfs}
\usepackage{pgf,tikz}
\usepackage{bbm}
\usepackage{natbib}
\usepackage{authblk}
\usepackage{subfigure}
\usepackage[bookmarks={true},bookmarksopen={true}]{hyperref}
\usepackage{xspace}

\let\oldabstract\abstract
\makeatletter
\renewcommand\abstract{%
  \providecommand\keywords{\par\medskip\noindent\textit{Keywords:}\xspace}
  \oldabstract\noindent\ignorespaces}
\makeatother

\newcommand{\EE}{\mathbb{E}}
\newcommand{\R}{\mathbb{R}}
\newcommand{\bfX}{X}
\newcommand{\dd}{\mathrm{d}}
\newcommand{\ta}{\theta}

\newcommand{\diam}{\mbox{\scriptsize diam}}
\newcommand{\C}{C}

\newcommand{\rhotwo}{\rho^{(2)}(u,v;\ta)}

\newtheorem{prop}{Proposition}[section]
\newtheorem{lem}[prop]{Lemma}
\newtheorem{rmq}[prop]{Remark}

\newtheorem{theo}[prop]{Theorem}
\newcommand{\cara}[1]{\mathbbm{1}_{#1}}
\newcommand{\der}{\mathrm{d}}
\newcommand{\var}{\mbox{Var}}
\newcommand{\E}{\mathbb{E}}
\newcommand{\cvlaw}{\overset{\mathcal{L}}{\longrightarrow}}

\renewcommand{\epsilon}{\varepsilon}

\usepackage{fancyhdr}
\graphicspath{{.}{../Fig/}{../../../../../Dropbox/DPPEE/Text/Fig/}}

\title{Adaptive estimating function inference for non-stationary determinantal point processes}

\author[1]{Fr\'ed\'eric Lavancier}
\author[2]{Arnaud Poinas}
\author[3]{Rasmus Waagepetersen}
\affil[1]{Laboratoire de Math\'ematiques Jean Leray, University of Nantes.}
\affil[2]{IRMAR, University of Rennes 1.}
\affil[3]{Department of Mathematical Sciences, Aalborg University.}

\date{}

\begin{document}
\maketitle

\begin{abstract} 
Estimating function inference is indispensable for many common point process models where the joint intensities are tractable while the likelihood function is not. In this paper we establish asymptotic normality of estimating function estimators in a very general setting of non-stationary point processes. We then adapt this result to the case of non-stationary determinantal point processes which are an important class of models for repulsive point patterns. In practice often first and second order estimating functions are used. For the latter it is common practice to omit contributions for pairs of points separated by a distance larger than some truncation distance which is usually specified in an ad hoc manner. We suggest instead a data-driven approach where the truncation distance is adapted automatically to the point process being fitted and where the approach integrates seamlessly with our asymptotic framework. The good performance of the adaptive approach is illustrated via simulation studies for non-stationary determinantal point processes and by an application to a real dataset.
 
 \keywords asymptotic normality, determinantal point processes, estimating functions, joint intensities, non-stationary, repulsive. 
 \end{abstract}

\section{Introduction}

A common feature of spatial point process models (except for the
Poisson process case) is that the likelihood function is not available
in a simple form. Numerical approximations of the likelihood function
are available \citep[see e.g.\ ][for
reviews]{moeller:waagepetersen:03,moeller:waagepetersen:07} but the
approaches are often computationally demanding and the distributional
properties of the approximate maximum likelihood estimates may be
difficult to assess. Therefore much work has focused on establishing
computationally simple estimation methods that do not require
knowledge of the likelihood function. 

In this paper we focus on estimation methods for point processes which have known joint intensity functions. This includes many cases of Cox and cluster point process models \cite[][]{moeller:waagepetersen:03,illian:etal:08,baddeley:rubak:turner:15} as well as determinantal point processes \cite[][]{macchi:75,soshnikov:00,shirai:takahashi:03,Lavancier15}. These classes of models are quite different since realizations of Cox and cluster point processes are aggregated while determinantal point processes produce regular point pattern realizations. 

Knowledge of an $n$th order joint intensity enables the use of the
so-called Campbell formulae for computing expectations of statistics
given by random sums indexed by $n$-tuples of distinct points in a
point process. Unbiased estimating functions can then be constructed
from such statistics by subtracting their expectations. So far mainly
the cases of first and second order joint intensities have been
considered where the first order joint intensity is simply the intensity
function. %However, consideration of higher order estimating functions may be worthwhile to obtain more precise estimators or to identify parameters in complex point process models.

Theoretical results have been established in a variety of special
cases of first and second order estimating functions for Cox and
cluster processes 
\citep{schoenberg:05,guan:06,waagepetersen:07,guan:loh:07,waagepetersen:guan:09}
and for the closely related Palm likelihood estimators
\citep{tanaka:ogata:stoyan:08,prokesova:jensen:13,prokesova:dvorak:jensen:16}. The
common general structure of the estimating functions on the other hand
calls for a general theoretical set-up which is the first contribution
of this paper. Our set-up also covers third or higher order estimating
functions and combinations of such estimating
functions, providing a general unifying framework.

The literature on statistical inference for continuous determinantal
point processes is quite limited. A Bayesian approach is considered in \cite{affandi2014learning}, while likelihood and minimum contrast estimation methods are discussed in \cite{Lavancier15}. In fact, maximum likelihood is not feasible in general and only an approximated version of the likelihood is proposed in the stationary case in \cite{Lavancier15}, without theoretical guarantees.  On the other hand consistency and asymptotic normality of minimum contrast estimators based on the pair correlation function or the Ripley's $K$ function have been established for 
determinantal point processes in \cite{biscio:lavancier:17}, but only in the stationary case. Based on
the general set-up our second main contribution is to provide a
detailed theoretical study of estimating function estimators for general non-stationary determinantal point processes.

Specializing to second-order  estimating functions, a common approach \citep{guan:06,tanaka:ogata:stoyan:08} is to restrict the random sum to pairs of $R$-close points for some user-specified $R>0$. This may lead to faster computation and improved statistical efficiency. The properties of the resulting estimators depend strongly on $R$ but only ad hoc guidance is available for the choice of $R$. Moreover, it is difficult to account for ad hoc choices of $R$ when establishing theoretical results. Our third contribution is a simple intuitively appealing adaptive choice of $R$ which leads to a theoretically tractable estimation procedure. We demonstrate its usefulness in simulation studies for determinantal point processes as well as an example of a cluster process. The practical advantage of the adaptive choice is further illustrated by an application to a dataset of locations of Japanese pines.

\section{Estimating functions based on joint intensities}

A point process $\bfX$ on $\R^d$, $d \ge 1$, is a locally finite
random subset of $\R^d$. For $B\subseteq \R^d$, we let $N(B)$ denote
the random number of points in $\bfX \cap B$ and $|B|$  the Lebesgue measure of $B$. That $\bfX$ is locally
finite means that $N(B)$ is finite almost surely whenever $B$ is
bounded. The so-called joint intensities of a point process are
described in Section~\ref{sec:joint}. In this paper we mainly focus on
determinantal point processes, detailed  in Section~\ref{sec:asymptoticres}. A prominent feature of determinantal point processes is that they have known joint intensity functions of any order.

\subsection{Joint intensity functions and Campbell formulae}\label{sec:joint}

For integer $n \ge 1$, the joint intensity $\rho^{(n)}$ of $n$th order is defined by
\begin{equation}\label{eq:joint} 
\EE \sum_{u_1,\ldots,u_n \in \bfX}^{\neq} \cara{u_1 \in B_1,\ldots,u_n \in B_n} = \int_{\times_{i=1}^n B_i} \rho^{(n)}(u_1,\ldots,u_n) \dd u_1 \cdots \dd u_n 
\end{equation}
for Borel sets $B_i \subseteq \R^d$, $i=1,\ldots,n$, assuming that the left hand side is absolutely continuous with respect to Lebesgue measure on $\R^d$. The $\neq$ over the summation sign means that the sum is over pairwise distinct points in $\bfX$. Of special interest are the cases $n=1$ and $n=2$ where the intensity function $\rho=\rho^{(1)}$ and the second order joint intensity $\rho^{(2)}$ determine the first and second order moments of the count variables $N(B)$, $B \subseteq \R^d$.
% is the intensity function and 
% assume that $\bfX$ has an intensity function $\rho$ and second-order
% joint intensity $\rho^{(2)}$ so that for bounded $A,B \subset \R^d$,
% \begin{align} \EE N(B) & = \int_B \rho(u) \dd u \nonumber\\
% \EE N(A) N(B) & = \int_{A
%   \cap B} \rho(u) \dd u + \int_A \int_B \rho^{(2)}(u,v) \dd u \dd
% v. \label{eq:moments}
% \end{align}
The pair correlation function $g(u,v)$ is defined as
\[ g(u,v) = \frac{\rho^{(2)}(u,v)}{\rho(u) \rho(v)} \]
whenever $\rho(u)\rho(v)>0$ (otherwise we define $g(u,v)=0$). The
product $\rho(u)g(u,v)$ can be interpreted as the intensity of $X$ at
$u$ given that $v \in X$. Hence $g(u,v)>1$ ($<1$) means that presence
of a point at $v$ increases (decreases) the likeliness of observing
yet another point at $u$.
% By \eqref{eq:joint} for $n=1,2$,
% \[
%   \Cov\big[ N(A), N(B) \big] = \int_{A \cap B} \rho(u) \dd u +
%   \int_{A}\int_{B} \rho(u)\rho(v)\big[ g(v,u) - 1 \big] \dd u\dd v
% \]
% for bounded $A,B \subset \R^d$.
% Hence, given the intensity function, the pair correlation function determines the covariances of count
% variables $N(A)$ and $N(B)$.
The Campbell formula
\[ \EE \sum_{u_1,\ldots,u_n \in \bfX}^{\neq} f(u_1,\ldots,u_n) = \int f(u_1,\ldots,u_n) \rho^{(n)}(u_1,\ldots,u_n) \dd u_1 \cdots \dd u_n \]
follows immediately from the definition of $\rho^{(n)}$ for any non-negative function $f: (\R^d)^n \rightarrow [0,\infty[$.

% The pair correlation function further has an infinitesimal interpretation: let $A$ denote an infinitesimally small region containing a location $u$ and
% of volume $\dd u$. Then $\rho(u)g(u,v)\dd u$ can be interpreted as
% the conditional probability of observing a point in $A$ given that
% $\bfX$ already has a point at the location $v$, see e.g.\
% \cite{coeurjolly:moeller:waagepetersen:16}. Moreover, $\rho(u)\dd u$ is the marginal probability of observing a
% point from $\bfX$ in $A$. Thus $g(u,v)>1$ ($<1$)
% implies that the presence of a point at $v$ yields an elevated
% (decreased) probability of observing yet another point at
% $u$. 
% % Assuming $g(u,v)=g(v-u)$ only depends on $u,v$ through $v-u$, the
% % $K$-function is defined by
% % \[ K(t) = \int_{b(0 \]
% We assume that $g$ is isotropic, i.e.\ with an abuse of notation, $g(u,v)=g(\|v-u\|)$. 

\subsection{A general asymptotic result for estimating functions}\label{sec:asympt_global}

Consider a parametric family of distributions $\{\mathbb{P}_{\theta}: \theta\in\Theta\}$ of point processes on $\R^d$, 
where $\Theta$ is a subset of $\R^p$. We assume a realization of the point process $X$ with distribution $\mathbb{P}_{\theta^*}$,
$\theta^*\in\mbox{Int}(\Theta)$, is observed on a bounded window $W_n\subset\R^d$.
We estimate the unknown parameter $\theta^*$ by the solution $\hat{\theta}_n$ of $e_n(\theta)=0$  (or one of the solutions if there are many)  where
\begin{equation*} \label{eq:eeGen}
e_n(\theta)= \left (
\begin{array}{c}
    \sum_{u_1,\cdots,u_{q_1}\in X\cap W_n}^{\neq}f_1(u_1,\cdots,u_{q_1};\theta)-\int_{W_n^{q_1}}f_1(u;\theta)\rho^{(q_1)}(u;\theta)\der u \\
    \vdots \\
    \sum_{u_1,\cdots,u_{q_l} \in X\cap W_n}^{\neq}f_l(u_1,\cdots,u_{q_l};\theta)-\int_{W_n^{q_l}}f_l(u;\theta)\rho^{(q_l)}(u;\theta)\der u \\
\end{array}
\right ) \end{equation*}
for $l$ given functions $f_i:(\R^d)^{q_i}\times \Theta \rightarrow \R^{k_i}$ such that $\sum_i k_i=p$. \\

A basic assumption for the following theorem 
%(verified in Appendix~\ref{sec:proof_general}) 
is  that a central limit theorem is available for $e_n(\theta^*)$ (assumption (\ref{Ass:CLT})). 
In addition to this, a number of technical assumptions~(\ref{Ass:reg}) through~(\ref{Ass:LimitMatrix}) (or (\ref{Ass:LimitMatrix}')), (\ref{Ass:intenseBound}) and (\ref{Ass:intense}) regarding existence and differentiability of joint intensities as well as differentiability of the $f_i$ are needed. All these conditions are detailed in Appendix \ref{sec:proof_general} while the proof of the  theorem is given in the supplementary material.

%
%\begin{theo} \label{th:general}
%Under Assumptions \eqref{Ass:reg} through \eqref{Ass:LimitMatrix} (or
%{\em(\ref{Ass:LimitMatrix}')}), \eqref{Ass:intenseBound} and
%\eqref{Ass:intense}, with a probability tending to one as
%$n\rightarrow\infty$, there exists a sequence of roots $\hat
%\ta_n$ of the estimating equations $e_n(\theta)=0$ for which
% \[\hat{\theta}_n\cvproba\theta^*.\]
%Moreover, if \eqref{Ass:CLT} holds true, then
%\[|W_n|\Sigma_n^{-1/2}H_n(\theta^*)(\hat{\theta}_n-\theta^*)\cvlaw\mathcal{N}(0,I_p).\]
%where $\Sigma_n=\var(e_n(\theta^*))$, $H_n(\theta^*)$ is defined in \eqref{Ass:LimitMatrix}, and $I_p$ is the $p\times p$ identity matrix.
%\end{theo}

\begin{theo} \label{th:general}
Under Assumptions \eqref{Ass:reg} through \eqref{Ass:LimitMatrix} (or
{\em(\ref{Ass:LimitMatrix}')}), \eqref{Ass:intenseBound} and
\eqref{Ass:intense}, with a probability tending to one as
$n\rightarrow\infty$, there exists a $|W_n|^{1/2}$-consistent sequence of roots $\hat
\ta_n$ of the estimating equations $e_n(\theta)=0$. Precisely,
 for all $\epsilon>0$, there exists $A>0$ such that
\begin{equation*}
\mathbb{P}(\exists \hat\theta_n: e_n(\hat\theta_n)=0~\mbox{and}~|W_n|^{1/2}\, \|\hat\theta_n-\theta^*\|<A)>1-\epsilon
\end{equation*}
for a sufficiently large $n$.

Moreover, if \eqref{Ass:CLT} holds true, then
\[|W_n|\Sigma_n^{-1/2}H_n(\theta^*)(\hat{\theta}_n-\theta^*)\cvlaw\mathcal{N}(0,I_p),\]
where $\Sigma_n=\var(e_n(\theta^*))$, $H_n(\theta^*)$ is defined in \eqref{Ass:LimitMatrix}, and $I_p$ is the $p\times p$ identity matrix.
\end{theo}

\begin{rmq}
While the parameter $\theta^*$ is generally uniquely defined (in the sense that $\theta\mapsto \mathbb{P}_{\theta}$ is injective) and verifies $\E(e_n(\theta^*))=0$, the solution to $e_n(\theta)=0$ may not be unique. The above theorem states that there exists a consistent and asymptotically Gaussian sequence of solutions, but unicity is not guaranteed. This drawback is unfortunately common in most asymptotic results for estimating functions inference, see the references in introduction, \cite{sorensen99},  the handbook by  \cite{heyde1997} or the discussion in \cite[Section~5.6]{van2000}. 
Nonetheless it can be proved that the solution is unique for $n$ sufficiently large  whenever $\lim_{n\to\infty} e_n(\theta)/|W_n|$ admits a unique zero, see~\cite{jacod2017}. 
But simple sufficient conditions to ensure the latter condition elude us.
\end{rmq}

\subsection{Second order estimating functions}\label{sec:secondorder}

Referring to the previous section, much attention has been devoted to instances of the case $l=1$, $q_1=2$ and $k_1=p$. In this case we obtain a second-order estimating function  of the form 
\begin{equation}\label{eq:ef}  e_n(\theta)= \sum_{u,v \in \bfX \cap W_n}^{\neq} f(u,v;\ta) - \int_{W_n^2} f(u,v;\ta) \rho^{(2)}(u,v;\ta) \dd u \dd v. 
\end{equation}
In \cite{guan:06}, the author noted that for computational and statistical efficiency it may be advantageous to use only close pairs of points rather than all pairs of points. Thus in \eqref{eq:ef} it is common practice to introduce an indicator $\cara{\|u-v\| \le R}$ for some constant $0<R$ or choose $f$ so that $f(u,v)=0$ whenever $\|u-v\|>R$. We discuss a method for choosing $R$ in Section~\ref{sec:adaptiveR}. 

The general form \eqref{eq:ef} includes e.g.\ the score functions of second-order composite likelihood \citep{guan:06,waagepetersen:07} and Palm likelihood functions \citep{tanaka:ogata:stoyan:08,prokesova:jensen:13,prokesova:dvorak:jensen:16} as well as score functions of minimum contrast object functions based on non-parametric estimates of summary statistics as the Ripley's $K$ or the pair correlation function. For the second-order composite likelihood defined in equation (4) in \cite{guan:06},
\[ f(u,v;\ta)= \frac{\nabla_\theta \rho^{(2)}(u,v;\ta)}{\rhotwo}- \frac{\int_{W^2} \nabla_\theta \rho^{(2)}(u,v;\ta) \dd u \dd v}{ \int_{W^2} \rhotwo \dd u \dd v} \]
while 
\begin{equation}\label{eq:secondorder} f(u,v;\ta)= \frac{\nabla_\theta \rho^{(2)}(u,v;\ta)}{\rhotwo} \end{equation}
for the second-order composite likelihood proposed in \cite{waagepetersen:07}. The score of the Palm likelihood as generalized to the inhomogeneous case in \cite{prokesova:dvorak:jensen:16} is obtained with
\[ f(u,v;\ta) = \frac{\nabla_\theta \frac{\rhotwo}{\rho(u;\ta)}}{\rhotwo/\rho(u;\ta)} - \frac{1}{N(W)-1} \int_W \nabla_\theta\left( \frac{\rho^{(2)}(u,w;\ta)}{\rho(u;\ta)}\right) \dd  w.  \]
In \cite{prokesova:dvorak:jensen:16}, the authors also regarded the second-order composite likelihood proposed in \cite{waagepetersen:07} as a generalization of the stationary case Palm likelihood but the interpretation as a second-order composite likelihood given in \cite{waagepetersen:07} is more straightforward.

Considering a class of estimating functions of the form \eqref{eq:ef} a natural question is what is the optimal choice of $f$?  A solution to this problem is provided in \cite{deng:etal:16} where an approximation of the optimal $f$ is obtained by solving numerically a certain integral equation. This yields a statistically optimal estimation procedure but is computationally demanding and requires specification of third and fourth order joint intensities. When computational speed and ease of use is an issue, there is still scope for simpler methods. Moreover, given several (simple) estimation methods, it is possible to combine them adaptively in order to build a final estimator that achieves better properties than each initial estimator, see \cite{lavancier2016general,LR2017spatial}. 
 
\subsection{Adaptive version}\label{sec:adaptiveR}

Consider second-order composite likelihood using \eqref{eq:secondorder} but only $R$ close
pairs. The resulting weight function is then of the form
\begin{equation}\label{eq:CL_R} f_R(u,v;\ta)= \cara{\|u-v\| \le R} \frac{\nabla_\theta \rho^{(2)}(u,v;\ta)}{\rhotwo}. \end{equation}
As mentioned in the previous section, using only $R$ close pairs may be beneficial both for statistical efficiency and computational tractability. However, the possible improvement depends strongly on the
chosen $R$. Simulation studies such as in
\cite{prokesova:dvorak:jensen:16} and \cite{deng:etal:16} usually
compare results for several values of $R$ corresponding to different
multiples of some parameter associated with `range of
correlation'. For a cluster process this parameter could e.g.\ be the
standard deviation of the distribution for dispersal of offspring
around parents. For a determinantal point process the parameter would
typically be a correlation scale parameter in the kernel of the
determinantal point process, see Section~\ref{sec:asymptoticres}. In practice these parameters are not
known and among the quantities that need to be estimated. In \cite{guan:06} it is suggested to choose an $R$ that minimizes a goodness of fit criterion for the fitted point process model while the choice of $R$ in \cite{tanaka:ogata:stoyan:08} and \cite{waagepetersen:guan:09} is done by inspection of a non-parametric estimate of the pair correlation function (a similar appproach is suggested by \cite{heagerty:lele:98} and \cite{bai:etal:14} in the context of pairwise composite likelihood for random fields). Both approaches imply extra work and ad hoc decisions by the user and it becomes very complex to determine the statistical properties of the resulting parameter estimates. 

%A typical behaviour of many pair correlation functions is that 
 % $g(u,v;\ta)$ converges to a limiting value of 1 when $\|u-v\|$
  %increases and 
  %$|g(u,v;\ta)-1| \le |g(u,u;\ta)-1|$ where the upper bound does not depend on $u$.
%   If $g(u,v;\ta)=1$ for $\|u-v\|>r_0$ then counts of points are uncorrelated when they are observed in regions separated by a distance of $r_0$. 
  
   A typical behaviour of many pair correlation functions is that 
  $g(u,v;\ta)$ converges to a limiting value of 1 when $\|u-v\|$
  increases and $|g(u,v;\ta)-1| \le M(u,v;\ta)$ where $$M(u,v;\ta)=\max_{s\,\in\{u,v\}} |g(s,s;\ta)-1|.$$ Note that for DPPs,  $M(u,v;\ta)=1$ (see the next section) and for stationary point processes, $M(u,v;\ta)$ does not depend on $u$ and $v$.
  If $g(u,v;\ta)=1$ for $\|u-v\|>r_0$ then counts of points are uncorrelated when they are observed in regions separated by a distance of $r_0$.

  Following the idea that $R$ should depend on some range property
  of the point process we therefore suggest to replace the constraint $\|u-v\|<R$ in  \eqref{eq:CL_R} by the constraint \[\frac{|g(u,v;\ta)-1|}{M(u,v;\ta)}>\epsilon,\]
for a small $\epsilon$. If e.g.\ $\epsilon=1\%$ this means that we only consider pairs of points $(u,v)$ so that the difference between $g(u,v;\ta)$ and the limiting value $1$ is within 1\% of the maximal value $M(u,v;\ta)$. Note that this choice of pairs of points is adaptive in that it depends on $\ta$. 

We then modify the function $f_R$ to be
\begin{equation}\label{eq:f_adapt} f_{\text{adap}}(u,v;\ta)=
 w \left (\epsilon \frac{ M(u,v;\ta)}{g(u,v;\ta)-1} \right ) \frac{\nabla_\theta
    \rho^{(2)}(u,v;\ta)}{\rhotwo} \end{equation}
   where $w$ is some weight function of bounded support
$[-1,1]$. Later on, when establishing asymptotic results, we will also
assume that $w$ is differentiable. A common example of admissible weight function is $w(r)=e^{1/(r^2-1)}$ for $-1\leq r\leq 1$, while $w(r)=0$ otherwise. The user needs to specify a value of $\epsilon$ but in contrast to the original tuning parameter $R$, $\epsilon$ has an intuitive meaning independent of the underlying point process. We  choose $\epsilon=1\%$. % In the simulation study in Section~\ref{sec:simu}  we also consider $\epsilon=5\%$ in order to investigate the sensitivity to the choice of $\epsilon$. 

We emphasize that choosing $\epsilon=1\%$ is not necessarily optimal. An optimal $\epsilon$ might be found by maximizing the Godambe information as a function of $\epsilon$ but this is not straightforward and the computational advantages of our approach would be lost. In fact, if Godambe optimality is key, we suggest to consider the previously mentioned approach by \cite{deng:etal:16} to identify an optimal second order estimating function.

\section{Asymptotic results for determinantal point processes}\label{sec:asymptoticres}
A point process $X$ is a determinantal point process (DPP for short) with kernel $K:\R^d \times \R^d \rightarrow \R$ if for all $n \ge 1$, the joint intensity $\rho^{(n)}$ exists and is of the form
\[ \rho^{(n)}(u_1,\ldots,u_n) = \text{det}[K](u_1,\ldots,u_n) \]
for all $\{u_1,\ldots,u_n \} \subset \R^d$, where $[K](u_1,\ldots,u_n)$ is the matrix with entries $K(u_i,u_j)$. The intensity function is thus $\rho(u)=K(u,u)$, $u \in \R^d$. If a determinantal point process with kernel $K$ exists it is unique. General conditions for existence are presented in \cite{Lavancier15}. In particular, if $K$ admits the form 
\begin{equation}\label{dppSOIRS} K(u,v)=\sqrt{\rho(u)\rho(v)}C(u-v) \end{equation}
for a function $C: \R^d \rightarrow \R$ with $C(0)=1$, then a sufficient condition for existence of a DPP with kernel $K$ is that $\rho$ is bounded and that $C$ is a square integrable continuous covariance function with spectral density bounded by $1/\|\rho\|_{\infty}$. The normalization $C(0)=1$ ensures that $\rho$ is the intensity of the DPP. 

We now consider a parametric family of DPPs on $\R^d$ with kernels $K_{\theta}$ where $\theta\in\Theta$ and $\Theta \subseteq \mathbb R^p$ \citep[see][for examples of such families]{Lavancier15, Biscio16}. Henceforth, we assume that $K_\ta$ is symmetric, continuous and the DPP with kernel $K_\ta$ exists for all $\ta \in \Theta$. 
Note that in general, it is possible that two different kernels generate the same DPP distribution. This identifiability issue especially arises  in the case of a discrete state space, where the distribution of  a DPP is only identified up to flips of the signs of the rows and columns of its matrix kernel (see \cite{engel1980} or \cite{rising2015}). However, in the continuous case, corresponding to our framework, the kernel of a DPP is uniquely determined whenever the intensity function is positive, see Proposition~\ref{supp-Ident} and its corollary in the supplementary material. Assuming a positive intensity function is not restrictive for statistical applications of DPPs.

An expression for the likelihood of a
DPP on a bounded window is provided in \cite{Lavancier15}, where
likelihood based inference for stationary DPPs is discussed. However, the expression depends on a spectral representation
of $K$ which is rarely known in practice and must be approximated
numerically. Letting $n$ denote the number of observed points, the
likelihood further requires the computation of an $n\times n$ dense
matrix which can be time consuming for large $n$. As an alternative, minimum contrast estimation is considered in \cite{Biscio16}, based on the pair correlation function or Ripley's $K$-function, but only for stationary DPPs. In the following, we consider general non-stationary DPPs and the estimator $\hat \theta_n$ obtained by solving $e_n(\theta)=0$ where $e_n$ is given by \eqref{eq:ef}. Note that the distribution for any classical parametric DPP model  \cite[showcased in][]{Lavancier15, Biscio16} is uniquely determined by its first two order intensity functions, in the sense that $\theta\mapsto (\rho(.;\theta),\rho^{(2)}(.,.;\theta))$ is injective. This justifies the use of second order estimating functions for DPPs.

We establish in Section \ref{sec:asymptDPP_global} using Theorem \ref
{th:general} the asymptotic properties of the estimate $\hat \theta_n$ where $e_n$ is given by \eqref{eq:ef} for a wide class of test functions $f$. 
In Section \ref{sec:asympt_sep}, we focus on a particular case of the
DPP model, where the parameter $\theta=(\beta,\psi)$  can be separated into a parameter $\beta$ only appearing in the intensity function and a
parameter $\psi$ only appearing in the pair correlation function.
Following \cite{waagepetersen:guan:09}, it is natural to consider a
two-step estimation procedure where in a first step $\beta$ is
estimated by a Poisson likelihood score estimating function, which provides a consistent estimate of the intensity  \cite{schoenberg:05}, and in a
second step the remaining parameter $\psi$ is estimated by a second
order estimating function as in \eqref{eq:ef}, where $\beta$ is
replaced by $\hat\beta_n$ obtained in the first step. The asymptotic
properties of this two-step procedure again follow as a special case of Theorem \ref{th:general}.

\subsection{Second order estimating functions for DPPs}\label{sec:asymptDPP_global}

In this part and in the rest of the document, we consider the following notation. For any set $W\subset\R^d$ and $r>0$, we write $W\oplus r := \bigcup_{x\in W} B(x,r)$ and $W\ominus r := \{x\in W, B(x,r)\subset W\}$ for the dilation and erosion of the set $W$ where $B(x,r)$ denotes the ball centered in $x$ with radius $r$. 

We assume a realization of a DPP $X$ with kernel $K_{\theta^*}$, $\theta^*\in\mbox{Int}(\Theta)$, is observed on a bounded window $W_n\subset\R^d$. We estimate the unknown parameter $\theta^*$ by the solution $\hat{\theta}_n$ of $e_n(\theta)=0$ where $e_n(\theta)$ is given by \eqref{eq:ef} for a given $\R^p$-valued function $f$. Therefore, we are in a special case of the set-up in Section \ref{sec:asympt_global} with $l=1$, $q_1=2$, $k_1=p$ and we assume that $f_1=f$ satisfies  the assumptions \eqref{Ass:reg} through \eqref{Ass:LimitMatrix} (or (\ref{Ass:LimitMatrix}')) listed in Appendix~\ref{sec:proof_general}.
The condition \eqref{Ass:reg} in this case demands that $\theta\mapsto f(u,v;\theta)$ is twice continuously differentiable in a  neighbourhood of $\theta^*$ and for $\ta$ in this neighbourhood, the derivatives are bounded with respect to $(u,v)$ uniformly in $\ta$. Moreover, from \eqref{Ass:fbound}, there exists $R>0$ such that  for all $\theta$ in a neighbourhood of $\theta^*$,
\begin{equation}\label{ass_f_tau}f(u,v;\theta)=0\quad\text{if}\quad \|u-v\|>R.\end{equation}
Concerning  \eqref{Ass:LimitMatrix} (or (\ref{Ass:LimitMatrix}')),  this condition controls the asymptotic behaviour of the matrix $H_n(\theta)$ given by
\begin{equation*}
H_n(\theta)=\frac{1}{|W_n|}\int_{W_n^{2}} f(u,v;\ta)\nabla_\theta\rho^{(2)}(u,v;\ta)^T \der u\der v,
\end{equation*}
where we recall that in this setting 
\begin{equation}\label{rho2DPP}\rho^{(2)}(u,v;\ta)=K_{\theta}(u,u)K_{\theta}(v,v)-K_{\theta}(u,v)^2.\end{equation}

The assumptions \eqref{Ass:LimitMatrix} and (\ref{Ass:LimitMatrix}') are technical and needed for the  consistency  of the estimation procedure. 
%They are nonetheless generally verified when $f$ is defined as in \eqref{eq:f_adapt} even if $X$ is not stationary,  as detailed in Lemmas~\ref{F3stat} and \ref{F3nonstat}. 
When $H_n$ is a symmetric matrix, assumption  \eqref{Ass:LimitMatrix} seems simpler to verify than (\ref{Ass:LimitMatrix}'). 
As an important example,  when $f$ is defined as in \eqref{eq:f_adapt},  we prove in Lemmas~\ref{F3stat} and \ref{F3nonstat} that  \eqref{Ass:LimitMatrix} is generally satisfied even if $X$ is not stationary.

Finally, as shown in the proof of Theorem \ref{th:generalDPP} below, the assumptions \eqref{Ass:intenseBound} through \eqref{Ass:CLT} in Theorem \ref{th:general} are implied by the following:
%
%We recall that in this setting 
%\[\rho^{(2)}(u,v;\ta)=K_{\theta}(u,u)K_{\theta}(v,v)-K_{\theta}(u,v)^2.\]
% Moreover, as shown in the proof of Theorem \ref{th:generalDPP} below, the assumptions \eqref{Ass:intenseBound} through \eqref{Ass:CLT} in Theorem \ref{th:general} become
%
%
%
%
%
%We assume a realization of a DPP $X$ with kernel $K_{\theta^*}$, $\theta^*\in\mbox{Int}(\Theta)$, is observed on a window $W_n\subset\R^d$. We estimate the unknown parameter $\theta^*$ by the solution $\hat{\theta}_n$ of $e_n(\theta)=0$ where $e_n(\theta)$ is given by \eqref{eq:ef} for a given $\R^p$-valued function $f$. Therefore, we are in a special case of the set-up in Section \ref{sec:asympt_global} with $l=1$, $q_1=2$, $k_1=p$ and $f_1=f$. We recall that in this setting 
%\[\rho^{(2)}(u,v;\ta)=K_{\theta}(u,u)K_{\theta}(v,v)-K_{\theta}(u,v)^2.\]
% Moreover, as shown in the proof of Theorem \ref{th:generalDPP} below, the assumptions \eqref{Ass:intenseBound} through \eqref{Ass:CLT} in Theorem \ref{th:general} become
\begin{enumerate}[label=(D\arabic*),ref=D\arabic*]

\item $\theta\mapsto K_{\theta}(u,v)$ is twice continuously differentiable in a neighborhood of $\theta^*$, for all $u,v\in\R^d$. Moreover, the first and second derivative of $K_\theta$ with respect to $\theta$ are bounded with respect to $u,v\in\R^d$ uniformly in $\theta$ in a neighborhood of $\theta^*$.
\label{Ass:regDPP}

\item The kernel $K_{\theta^*}$ satisfies, for some $\epsilon>0$,
\[\sup_{\|u-v\|>r}K_{\theta^*}(u,v)=o(r^{-(d+\epsilon)/2}).\]
\label{Ass:Kernel}
 
\item $\liminf_{n}\lambda_{\min}(|W_n|^{-1}\Sigma_n)>0$ where $\Sigma_n:=\var(e_n(\theta^*))$ and $\lambda_{\min}(|W_n|^{-1}\Sigma_n)$ denotes the smallest eigenvalue of $|W_n|^{-1}\Sigma_n$. \label{Ass:VarControl}
\end{enumerate}

\begin{enumerate}[label=(W),ref=W]
\item $\exists\epsilon>0~\mbox{s.t.}$~$|\partial
  W_n\oplus(R+\epsilon)|=o(|W_n|)$, where $\partial$ in this context denotes the boundary of a set, $R$ is defined in \eqref{ass_f_tau},  and $|W_n|\to\infty$, as
  $n\to\infty$. \label{Ass:ShapeWindow}
\end{enumerate}

Let us briefly comment on these assumptions. \eqref{Ass:regDPP} is a standard regularity assumption. Condition \eqref{Ass:Kernel} is not restrictive since all standard parametric kernel families  satisfy $\sup_{\|u-v\|>r}K_{\theta}(u,v) = O(r^{-(d+1)/2})$, including the most repulsive stationary DPP \citep[see][]{Lavancier15,Biscio16}. Condition \eqref{Ass:VarControl} ensures that the asymptotic variance in the central limit theorem below is not degenerated. Finally, Assumption \eqref{Ass:ShapeWindow} makes specific the fact that $W_n$ is not too irregularly shaped and is not bounded in any direction. It is for instance fulfilled if $W_n$ is a Cartesian product of $d$ intervals whose lengths tends to infinity.

\begin{theo} \label{th:generalDPP}
Under Assumptions \eqref{Ass:regDPP} and \eqref{Ass:Kernel}, if assumptions \eqref{Ass:reg} through \eqref{Ass:LimitMatrix} (or {\em(\ref{Ass:LimitMatrix}')}) are satisfied for $f_1=f$, with a probability tending to one as
$n\rightarrow\infty$, there exists a $|W_n|^{1/2}$-consistent sequence of roots $\hat
\ta_n$ of the estimating equations $e_n(\theta)=0$.
If moreover \eqref{Ass:ShapeWindow} and \eqref{Ass:VarControl} holds true, then 
\[|W_n|\Sigma_n^{-1/2}H_n(\theta^*)(\hat{\theta}_n-\theta^*)\cvlaw\mathcal{N}(0,I_p).\]
\end{theo}
\begin{proof} We deduce from \eqref{rho2DPP} that \eqref{Ass:regDPP} implies \eqref{Ass:intenseBound}. Moreover, it was shown in \cite{poinas17} that \eqref{Ass:intense} is a consequence of \eqref{Ass:Kernel} and that \eqref{Ass:CLT} is a consequence of \eqref{Ass:Kernel}, \eqref{Ass:VarControl} and \eqref{Ass:ShapeWindow}. Thus, we can conclude by applying Theorem \eqref{th:general} in the case $l=1$ and $q_1=2$.
\end{proof}

In the case of a stationary $X$ and $f$ given by \eqref{eq:f_adapt}, the following lemma shows that \eqref{Ass:LimitMatrix} is satisfied under mild assumptions that are violated only in degenerate cases. For instance, if $p=1$, the last assumption boils down to $\nabla_\theta \rho^{(2)}(0,t;\ta^*) \neq 0$ for some $t \neq 0$ such that $|K_{\ta^*}(t)|>\sqrt{\epsilon}K_{\ta^*}(0)$.
 In particular it is not difficult to verify these assumptions for the stationary parametric kernels considered in our simulation study of Section~\ref{sec:simu}, namely the Bessel-type and the Gaussian kernels, see the supplementary material.

\begin{lem} \label{F3stat}
Assume \eqref{Ass:ShapeWindow} and \eqref{Ass:Kernel}, suppose  $X$ is stationary and let $f$ be as in \eqref{eq:f_adapt}. 
Assume that $w$ is positive on $[0,1)$, vanishes on $[1,\infty)$ and is differentiable on $\R_+$. If $t\mapsto f(0,t;\ta^*)$ is integrable and ${\rm span}\{\nabla_\theta\rho^{(2)}(0,t;\ta^*)  : |K_{\ta^*}(t)| > \sqrt{\epsilon}K_{\ta^*}(0) \}=\R^p$, then  
\eqref{Ass:LimitMatrix} is satisfied.
\end{lem}
\begin{proof}
By definition of $w$ and \eqref{Ass:Kernel}, there exists $R>0$ such that $f(0,t;\theta^*)=0$ when $\|t\|\geq R$. By Lemma \ref{lem_general}, since $t\mapsto f(0,t;\ta^*)$ is integrable then $H_n(\theta^*)$ converges towards the positive semi-definite matrix $H(\theta^*)=\int_{\|t\|< R}h(t)\der t$ where the function $h:\R^d\to \R^{p\times p}$ is defined by
\[h(t) = w \left ( \frac{\epsilon K_{\ta^*}(0)^2}{K_{\ta^*}(t)^2} \right )\frac {\nabla_\theta\rho^{(2)}(0,t;\ta^*) \nabla_\theta\rho^{(2)}(0,t;\ta^*)^T}{\rho^{(2)}(0,t;\ta^*) } .\]
In this case, proving \eqref{Ass:LimitMatrix} is equivalent to showing that $\phi^T H(\theta^*) \phi=0$ only if $\phi=0$. For this, let $A$ be the set of $t$ such that $|K_{\ta^*}(t)|>\sqrt{\epsilon}K_{\ta^*}(0)$, $\phi\in\R^p$ and note that since $w( \epsilon K_{\ta^*}(0)^2/K_{\ta^*}(t)^2)>0$ for $t \in A$ and $h(t)$ is continuous and positive semi-definite,
\begin{align*}
\phi^T H(\ta^*)\phi=0 \quad&\Leftrightarrow\quad \forall t\in A,\ \phi^Th(t)\phi=0\\
&\Leftrightarrow\quad \forall t\in A,\ \nabla_\theta\rho^{(2)}(0,t;\ta^*)^T \phi =0\\
&\Leftrightarrow\quad \phi\in \left({\rm span}\{\nabla_\theta\rho^{(2)}(0,t;\ta^*) \, : \, t\in A\}\right)^{\perp}.
 \end{align*} 
By assumption ${\rm span}\{\nabla_\theta\rho^{(2)}(0,t;\ta^*) \, : \,  t\in A\}=\R^p$ whereby $\phi=0$, which  concludes the proof. 
\end{proof}

Similarly, we can show that even in the non-stationary case, condition \eqref{Ass:LimitMatrix} is satisfied for the function in \eqref{eq:f_adapt} but under slightly stronger assumptions on $\nabla_\theta\rho^{(2)}(u,v;\ta^*)$.
Namely, we demand  that all functions  $v\mapsto\nabla_\theta\rho^{(2)}(u,v;\ta^*)$ are not contained  in a single hyperplane of $\R^p$ nor confined around $0$. This is similar in essence to what we have assumed in the previous corollary but with the need of a uniform condition with respect to $u$. Functions that do not satisfy these requirements are arguably degenerate. In particular, a straightforward calculus carried out in the supplementary material shows that the non-stationary Bessel-type kernel used in our simulation study satisfies these assumptions.

\begin{lem} \label{F3nonstat}
Assume \eqref{Ass:ShapeWindow}, \eqref{Ass:Kernel} and that $K_{\theta^*}$ is bounded. Let $f$ be as in \eqref{eq:f_adapt} and define $h:(\R^d)^2 \to \R^{p\times p}$ by
\[h(u,v) =w \left ( \frac{\epsilon K_{\ta^*}(u,u)K_{\ta^*}(v,v)}{K_{\ta^*}(u,v)^2} \right ) \frac {\nabla_\theta\rho^{(2)}(u,v;\ta^*) \nabla_\theta\rho^{(2)}(u,v;\ta^*)^T}{\rho^{(2)}(u,v;\ta^*) } .\]
Assume that $w$ is positive on $[0,1)$, vanishes on $[1,\infty)$ and is differentiable on $\R_+$. If $\sup_{u \in \R^d}\|\int_{\R^d} h(u,v)\der v\|<+\infty$ and if there exists $\mu>1$ and $\delta>0$ such that for all $u\in\R^d$ and for all unit vectors $\phi$ of $\R^p$ there exists a subset $A$ of $\{v \, :\, K_{\ta^*}(u,v)^2>\mu\epsilon K_{\ta^*}(u,u)K_{\ta^*}(v,v) \}$ of positive Lebesgue measure $|A|>0$ and satisfying
$$\forall v\in A,\, |\phi^T\nabla_{\theta}\rho^{(2)}(u,v;\ta^*)|>\delta$$
then \eqref{Ass:LimitMatrix} is satisfied.
\end{lem}

\begin{proof}
By definition of $w$, \eqref{Ass:Kernel} and the fact that $K_{\theta^*}$ is bounded, there exists $R>0$ such that $h(u,v)=0$ when $\|v-u\|\geq R$.
 The integral in \eqref{Ass:LimitMatrix} writes
\[ H_n(\theta^*)=\frac{1}{|W_n|}\int_{W_n^2} h(u,v) \cara{\|u-v\| \le R}  \dd v \dd u =\frac{1}{|W_n|}\int_{W_n\ominus R}\int_{W_n}h(u,v) \cara{\|u-v\| \le R} \dd v \dd u + \epsilon_n  \]
where \[\epsilon_n = \frac{1}{|W_n|}\int_{W_n\setminus (W_n\ominus R)} \int_{W_n} h(u,v) \cara{\|u-v\| \le R} \dd v \dd u.\]
By \eqref{Ass:ShapeWindow}, we have
\begin{align*}
\|\epsilon_{n}\| &\leq \frac{|W_n\setminus W_n\ominus R)|}{|W_n|}\sup_{u\in\R^d} \int_{\R^d} \|h(u,v)\|\der v\\
& \le \frac{|\partial W_n \oplus R|}{|W_n|}\sup_{u\in\R^d} \int_{\R^d}\|h(u,v)\|\der v\to 0,
\end{align*} 
and for all $\phi$,
\[\phi^T\left(\int_{W_n\ominus R}\int_{W_n}h(u,v) \cara{\|u-v\| \le R} \dd u \dd v \right)\phi
 =  \int_{W_n\ominus R}\left(\int_{\|u-v\| \le R}\phi^T h(u,v)\phi \dd v \right)\dd u. \]
By our assumption on $\nabla_{\theta}\rho^{(2)}$, there exists a set $A$ of positive Lebesgue measure such that
$$\forall v\in A,~|\phi^T\nabla_{\theta}\rho^{(2)}(u,v;\ta^*)|>\delta~~\mbox{and}~~w \left ( \frac{\epsilon K_{\ta^*}(u,u)K_{\ta^*}(v,v)}{K_{\ta^*}(u,v)^2} \right )>\inf_{x\in[0,1/\mu]}w(x).$$
Hence for $\|\phi\|=1$,
\begin{align*}
 &\frac{1}{|W_n|} \phi^T\left(\int_{W_n\ominus R}\int_{W_n}h(u,v) \cara{\|u-v\| \le R} \dd u \dd v \right)\phi\\
 \geq &\frac{\inf_{x\in[0,1/\mu]}w(x)}{|W_n|\|\rho^{(2)}(.,.;\theta^*)\|_{\infty}} \int_{W_n\ominus R}\left(\int_A|\phi^T\nabla_{\theta}\rho^{(2)}(u,v;\ta^*)|^2 \dd v \right)\dd u \\
 \geq &\frac{|W_n\ominus R| |A| \delta^2\inf_{x\in[0,1/\mu]}w(x)}{|W_n|\|\rho^{(2)}(.,.;\theta^*)\|_{\infty}}\\
  = &\left(\frac{|W_n|-|W_n\cap(\partial W_n \oplus R) |}{|W_n|}\right)\frac{|A|\delta^2\inf_{x\in[0,1/\mu]}w(x)}{\|\rho^{(2)}(.,.;\theta^*)\|_{\infty}}\\
 \rightarrow &\frac{|A|\delta^2\inf_{x\in[0,1/\mu]}w(x)}{\|\rho^{(2)}(.,.;\theta^*)\|_{\infty}}> 0
\end{align*}
where the limit is a consequence of \eqref{Ass:ShapeWindow}. Since the limit does not depend on $\phi$, then \eqref{Ass:LimitMatrix} is satisfied.
\end{proof}

\subsection{Two-step estimation for a separable parameter}\label{sec:asympt_sep}

We consider a family of kernels
\[K_{\theta}(u,v)=\sqrt{\rho(u;\beta)}\C(u,v;\psi)\sqrt{\rho(v;\beta)},\]
where $ \theta:=(\beta^T,\psi^T)^T \in\Theta\subset\R^{p+q}$ with
$\beta\in\R^p$ and $\psi\in\R^q$, $\rho(.;\beta)$ are non-negative
functions, and $\C(\cdot,\cdot;\psi)$ are correlation functions, in particular
$\C(u,u;\psi)=1$ for any $\psi$. % The set $\Theta$ is chosen such that
% a DPP exists for any $K_\theta$ in the family. 
Note that in this case the DPP with kernel $K_\theta$ has intensity $\rho(.;\beta)$ and its pair correlation function is $g(u,v;\psi)=1-C^2(u,v;\psi)$.

As in the preceding section, we assume a DPP $X$ with kernel $K_{\theta^*}$, $\theta^*\in\mbox{Int}(\Theta)$, is
observed on a bounded window $W_n\subset\R^d$.  In the spirit of \cite{waagepetersen:guan:09}, we estimate $\theta^*$ in two steps. First, $\beta^*$ is estimated as the solution $\hat\beta_n$ of $s_n(\beta)=0$ where
\begin{equation*}
s_n(\beta)=\sum_{u \in \bfX \cap W_n}\frac{\nabla_\beta\rho(u;\beta)}{\rho(u;\beta)}-\int_{W_n}\nabla_\beta\rho(u;\beta)\der u
\end{equation*}
is the score function for a Poisson point process.  Then, we estimate $\psi^*$ by the solution  $\hat{\psi}_n$ of $u_n(\hat\beta_n,\psi)=0$ where
\[u_n(\theta)= \sum_{u,v \in \bfX \cap W_n}^{\neq} f(u,v;\ta) - \int_{W_n^2} f(u,v;\ta) \rho^{(2)}(u,v;\ta) \dd u \dd v\]
for a given $\R^q$-valued function $f$ and where
$\rho^{(2)}(u,v;\ta)=\rho^{(2)}(u,v;\beta,\psi)= \rho(u;\beta)\rho(v;\beta)(1-C^2(u,v;\psi))$
in this case.  Here and in the following, for convenience of notation, we identify $u_n(\beta,\psi)$ with $u_n(\theta)$ when $\theta=(\beta^T,\psi^T)^T$.

This two-step procedure is a particular estimating equation procedure, since $\hat\theta_n:=(\hat\beta_n^T,\hat\psi_n^T)^T$ is obtained as the solution of $e_n(\theta)=0$ where $e_n(\theta)=(s_n(\beta)^T,u_n(\beta,\psi)^T)^T$.  Thus, this is a particular case of the setting in Section \ref{sec:asympt_global} where $l=2$, $q_1=1$, $q_2=2$, $f_1=\nabla_\beta\rho(u;\beta)/\rho(u;\beta)$ and $f_2=f$.

We assume in the following theorem the same conditions on the DPP $X$ as in the previous section. Similarly, we assume that \eqref{Ass:reg} through \eqref{Ass:LimitMatrix} (or (\ref{Ass:LimitMatrix}')) are satisfied for $f_1$ and $f_2$.  In this particular case, the matrix $H_n$ involved in \eqref{Ass:LimitMatrix}
simply writes
\begin{equation*}
H_n(\beta,\psi)=\left (
\begin{array}{cc}
    H^{1,1}_n(\beta,\psi) & 0\\
    H^{2,1}_n(\beta,\psi) & H^{2,2}_n(\beta,\psi)\\
\end{array}
\right )
\end{equation*}
where
\begin{align*}
& H^{1,1}_n(\beta)=\frac{1}{|W_n|}\int_{W_n} \frac{\nabla_\beta\rho(u;\beta)\nabla_\beta\rho(u;\beta)^T}{\rho(u;\beta)} \der u,\\
 & H^{2,1}_n(\beta,\psi)=\frac{1}{|W_n|}\int_{W_n^{2}} f(u,v;\beta,\psi) \nabla_\beta\rho^{(2)}(u,v;\beta,\psi)^T\der u\der v,\\
 & H^{2,2}_n(\beta,\psi)=\frac{1}{|W_n|}\int_{W_n^{2}} f(u,v;\beta,\psi) \nabla_\psi\rho^{(2)}(u,v;\beta,\psi)^T \der u\der v.
\end{align*}
Since it is a non symmetric matrix, condition (\ref{Ass:LimitMatrix}') is more  applicable than \eqref{Ass:LimitMatrix}. Mild conditions ensuring  (\ref{Ass:LimitMatrix}')  in the stationary case are provided in Lemma~\ref{lem_separable}.

\begin{theo} \label{th:sep}
Under Assumptions \eqref{Ass:regDPP} and \eqref{Ass:Kernel}, if
assumptions \eqref{Ass:reg} through \eqref{Ass:LimitMatrix} (or
{\em(\ref{Ass:LimitMatrix}')}) are satisfied for
$f_1=\nabla_\beta\rho(u;\beta)/\rho(u;\beta)$ and $f_2=f$, then
  with a probability tending to one as $n\rightarrow\infty$, there exists a $|W_n|^{1/2}$-consistent sequence of roots $\hat
\ta_n$ of the estimating equations $e_n(\theta)=0$. If moreover \eqref{Ass:ShapeWindow} and \eqref{Ass:VarControl} hold true, then 
\[|W_n|\Sigma_n^{-1/2} H_n(\theta^*)(\hat \theta_n - \theta^*)\cvlaw\mathcal{N}(0,I_{p+q}).\]
\end{theo}
\begin{proof}
The proof follows the same lines as the proof of Theorem \ref{th:generalDPP}.
\end{proof}

The next lemma is similar to Lemma~\ref{F3stat}. When $q=1$ the last technical condition boils down to  $\nabla_\psi (1-C^2(0,t;\psi^*))\neq 0$ for some $t$ such that $C(0,t;\psi^*)\geq\sqrt{\epsilon}C(0,0;\psi^*)$. In particular, the stationary kernels in Section~\ref{sec:simu} satisfy the required assumptions, see the supplementary material.

\begin{lem}\label{lem_separable}
Assume that for all $\theta$, $K_\theta(u,v)$ only depends on $u-v$, in which case $\rho(u;\beta)=\beta$ with $\beta>0$ and  $C(u,v;\psi)=C(0,v-u;\psi)$ with $\psi\in\R^q$.
Then the output of the first step is $\hat\beta_n = N(X\cap W_n)/|W_n|$. In the second step, assume
\begin{align*}
f(u,v;\beta,\psi)=& w \left (\frac{\epsilon}{1-g(u,v;\psi)} \right ) \frac{\nabla_\psi \rho^{(2)}(u,v;\beta,\psi)}{\rho^{(2)}(u,v;\beta,\psi)} \\
=&  w \left (\frac{\epsilon}{C(0,v-u;\psi)^2} \right )  \frac{\nabla_\psi (1-C^2(0,v-u;\psi))}{1-C^2(0,v-u;\psi)}.
\end{align*}
Assume that $w$ is positive on $[0,1)$, vanishes on $[1,\infty)$ and is differentiable on $\R_+$. If $t\mapsto f(0,t;\theta^*)$ is integrable and ${\rm span}\{\nabla_\psi (1-C^2(0,t;\psi^*)) \, : \, C(0,t;\psi^*)>\sqrt{\epsilon}\}=\R^q$, then {\em(\ref{Ass:LimitMatrix}')} is satisfied under \eqref{Ass:ShapeWindow}, \eqref{Ass:regDPP} and \eqref{Ass:Kernel}.
\end{lem}
\begin{proof}
By definition of $w$ and \eqref{Ass:Kernel}, there exists $R>0$ such that $f(0,t;\theta^*)=0$ when $\|t\|\geq R$. Since $K_\theta$ and $f$ are invariant by translation and $t\mapsto f(0,t;\theta^*)$ is integrable then $H_n(\theta)$ converges by Lemma~\ref{lem_general}. In particular, we have
\[H_n^{1,1}(\beta)\rightarrow\frac 1 {\beta},\]
\[H_n^{2,2}(\beta,\psi)\rightarrow{\beta}^2 \int_{\|t\|\leq R} h(t;\psi) \dd t,\]
\[H_n^{2,1}(\beta,\psi)\rightarrow 2\beta \int_{\|t\|\leq R} w \left (\frac{\epsilon}{C(0,t;\psi)^2} \right ) \nabla_\psi (1-C^2(0,t;\psi))\dd t,\]
where the function $h(.;\psi):\R^d\to \R^{p\times p}$ is defined by
\[h(t;\psi)=    w \left (\frac{\epsilon}{C(0,t;\psi)^2} \right )  \frac{\nabla_\psi (1-C^2(0,t;\psi))\nabla_\psi (1-C^2(0,t;\psi))^T}{1-C^2(0,t;\psi)}.\]
The limit of $H_n(\theta)$ is continuous by \eqref{Ass:regDPP}.  In this case, proving (\ref{Ass:LimitMatrix}') is equivalent to showing that the limit of $H_n(\theta^*)$ is invertible. Since this matrix is block triangular and $\beta>0$ then it is invertible if and only if the limit of $H_n^{2,1}(\theta^*)$ is invertible. This is done the same way as in Lemma \ref{F3stat}. 
\end{proof}

\section{Simulation study}\label{sec:simu}

%In this section we use simulation studies to investigate the performance of our adaptive estimating function and to compare two-step estimation with simultaneous estimation.

In this section we use simulation studies to investigate the performance of our adaptive estimating function. In Section~\ref{supp-sec:twostepvssimul} of the supplementary material, we additionally  compare two-step estimation, when it is feasible, with simultaneous estimation. Our recommendation is to use the two-step approach.

%\subsection{Performance of adaptive estimating function}\label{sec:performance}

In order to assess the adaptive test function \eqref{eq:f_adapt} against the truncated test function  \eqref{eq:CL_R} with a prescribed $R$, we consider a DPP model in $\R^2$ with a Bessel-type kernel
\[K(u,v) = \sqrt{\rho(u)\rho(v)} \,\frac{J_1(2\|u-v\|/\alpha)}{\|u-v\|/\alpha} ,\]
where $J_1$ denotes the Bessel function of the first kind, $\rho$ is the intensity and $\alpha$ controls the range of interaction of the DPP. For existence, $\rho$ and $\alpha$ must satisfy
\begin{equation}\label{existenceBessel} \alpha^2 \|\rho\|_{\infty}\leq \frac 1 \pi.\end{equation}
This relation shows the tradeoff between the expected number of points and the strength of repulsiveness that we can obtain. This model is a particular instance of the Bessel-type DPP introduced in \cite{Biscio16}. It covers a large range of repulsiveness, from the Poisson point process (when $\alpha$ is close to 0) to the most repulsive DPP (when $\alpha=1/\sqrt{\pi \|\rho\|_{\infty}}$).

For this model, we consider three constant values of $\rho$, $\rho\in\{50,100,1000\}$, corresponding to homogeneous DPPs,  and an inhomogeneous situation where $\rho(u)=\rho(x,y)=20\exp(4x)$ when $u\in [0,1]^2$. The latter case corresponds to a log-linear intensity function involving two parameters. For each $\rho$, three values of $\alpha$ are considered: a small one, a medium one, and a last one  close to the maximal possible value satisfying \eqref{existenceBessel}. Examples of point patterns simulated on  $[0,1]^2$ are displayed in Figure~\ref{fig:samples}. All simulations are carried out using \texttt{R} \citep{R}, in particular the library \texttt{spatstat} \citep{baddeley:rubak:turner:15}.

We estimate $\rho$ and $\alpha$ by a two-step procedure as studied in Section~\ref{sec:asympt_sep} from  realizations of the DPP on $W=[0,1]^2$. The alternative global approach of Section~\ref{sec:asymptDPP_global} is discussed in the next section. In the first step,  the parameters arising in $\rho$ are estimated by the score function for a Poisson point process. This gives $\hat\rho = N(X\cap W)/|W|$ in the homogeneous cases. In the second step, we consider the estimating equation based on  \eqref{eq:CL_R} where $\theta$ is $\alpha$ in this setting and  when $R\in\{0.05, 0.1,0.25\}$, and based on  the adaptive test function \eqref{eq:f_adapt} with $\epsilon=0.01$ and the weight function $w$ given at the end of Section~\ref{sec:adaptiveR}. This yields four different estimators of $\alpha$. 
The root mean square errors (RMSEs) of these estimators and the mean computation time estimated from 1000 replications are summarised in Table~\ref{tab:bessel}. Boxplots are displayed in Figure~\ref{supp-fig:bessel} in the supplementary material. Note that the codes have not been optimised, but the same computational strategy has been used for all methods, making the comparison of the mean computation time meaningful. 

The Bessel-type  kernel and the aforementioned test functions used in the two-step estimation procedure fulfill the assumptions of Theorem~\ref{th:sep} and Lemma~\ref{lem_separable} (for the homogeneous case), ensuring nice asymptotic properties of the estimators considered in this section. This is confirmed by the estimated RMSE's  reported in Table~\ref{tab:bessel}, that decrease when the intensity $\rho$ increases \cite[which  mimics the effect of an increasing window since rescaling the window by a factor $1/k$ is equivalent to change $\rho$ into $k^2\rho$ and $\alpha$ into $\alpha/k$, see (2.4) in][]{Lavancier15}. Moreover, these RMSE's show that the best choice of $R$ in the test function~\eqref{eq:CL_R} clearly depends on the range of interaction of the underlying process. This emphasizes the importance of a data-driven approach to choosing $R$ since the range is unknown in practice. Fortunately, the performance of the adaptive method is, except for the case $\rho=100,\alpha=0.01$, always better than the worst choice of $R$ and very close to the best $R$. For the exceptional case, the small differences in performance can be explained by Monte Carlo error. Further, use of the adaptive method implies only little or no extra computional effort. In presence of many points, the adaptive version is in fact much faster to compute than the estimator based  on \eqref{eq:CL_R} with the choice of a too large $R$, see for instance the results for $\rho=1000$ and $R=0.25$. 

Table~\ref{supp-tab:bessel2} in the supplementary material
  shows the root mean square errors of the adaptive estimator using $\epsilon=0.05$. The RMSEs obtained with $\epsilon=0.05$ are bigger than those obtained with $\epsilon=0.01$. Nevertheless, the adaptive method with $\epsilon=0.05$ still performs well in the sense that it usually performs better than the worst $R$ and usually almost as good as the best $R$.
Because the above estimation methods sometimes fail to converge, we also report in Table~\ref{supp-tab:percentcv} in the supplementary material  the percentages of times each method has converged in our simulation study. These percentages are similar for all methods. Note that the results in Table~\ref{tab:bessel} and in Figure~\ref{supp-fig:bessel} are based on 1000 simulations where all four methods have converged.

%\begin{figure}
%\begin{center}
%\subfigure[$\rho=50$, $\alpha=0.02$]{\includegraphics[scale=.33]{Xbessell50a2_sample.pdf}}
%$\;$ \subfigure[$\rho=50$, $\alpha=0.04$]{\includegraphics[scale=.33]{Xbessell50a4_sample.pdf}} 
%$\;$ \subfigure[$\rho=50$, $\alpha=0.07$]{\includegraphics[scale=.33]{Xbessell50a7_sample.pdf}} 
%\subfigure[$\rho=100$, $\alpha=0.01$]{\includegraphics[scale=.33]{Xbessell100a1_sample.pdf}}
%$\;$ \subfigure[$\rho=100$, $\alpha=0.03$]{\includegraphics[scale=.33]{Xbessell100a3_sample.pdf}} 
%$\;$ \subfigure[$\rho=100$, $\alpha=0.05$]{\includegraphics[scale=.33]{Xbessell100a5_sample.pdf}} 
%\subfigure[$\rho=1000$, $\alpha=0.005$]{\includegraphics[scale=.33]{Xbessell1000a005_sample.pdf}}
%$\;$ \subfigure[$\rho=1000$, $\alpha=0.01$]{\includegraphics[scale=.33]{Xbessell1000a010_sample.pdf}} 
%$\;$ \subfigure[$\rho=1000$, $\alpha=0.015$]{\includegraphics[scale=.33]{Xbessell1000a015_sample.pdf}} 
%$\;$ \subfigure[$\rho=\rho(u)$, $\alpha=0.005$]{\includegraphics[scale=.33]{Xbesselinhoma05_sample.pdf}}
%$\;$ \subfigure[$\rho=\rho(u)$, $\alpha=0.01$]{\includegraphics[scale=.33]{Xbesselinhoma1_sample.pdf}}
%$\;$ \subfigure[$\rho=\rho(u)$, $\alpha=0.015$]{\includegraphics[scale=.33]{Xbesselinhoma15_sample.pdf}} 
%\caption{Examples of point patterns simulated from a Bessel-type DPP on $[0,1]^2$ for different values of $\rho$ and $\alpha$.\label{fig:samples} }
%\end{center}
%\end{figure}

\begin{figure}[ht]
\begin{center}
{\small
\begin{tabular}{ccc}
% \includegraphics[scale=.3]{Xbessell50a2_sample.pdf} & \includegraphics[scale=.3]{Xbessell50a4_sample.pdf} & \includegraphics[scale=.3]{Xbessell50a7_sample.pdf}   \\
% $\rho=50$, $\alpha=0.02$ & $\rho=50$, $\alpha=0.04$ & $\rho=50$, $\alpha=0.07$ \\
 \includegraphics[scale=.3]{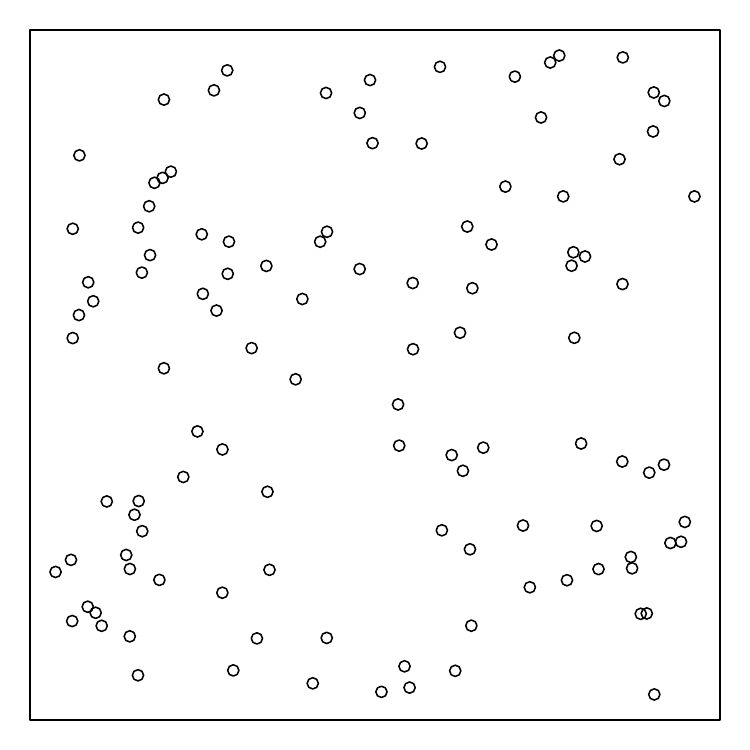} & \includegraphics[scale=.3]{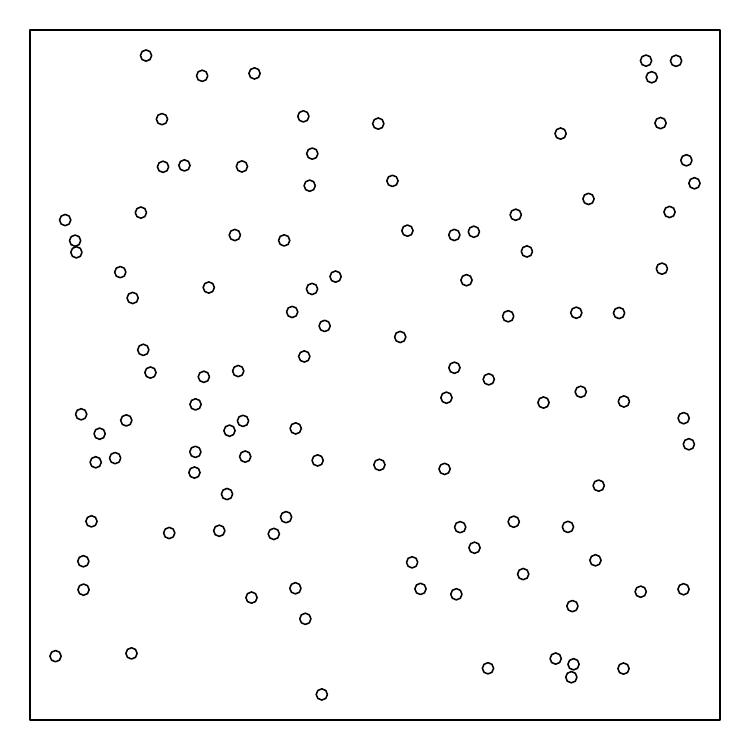} &  \includegraphics[scale=.3]{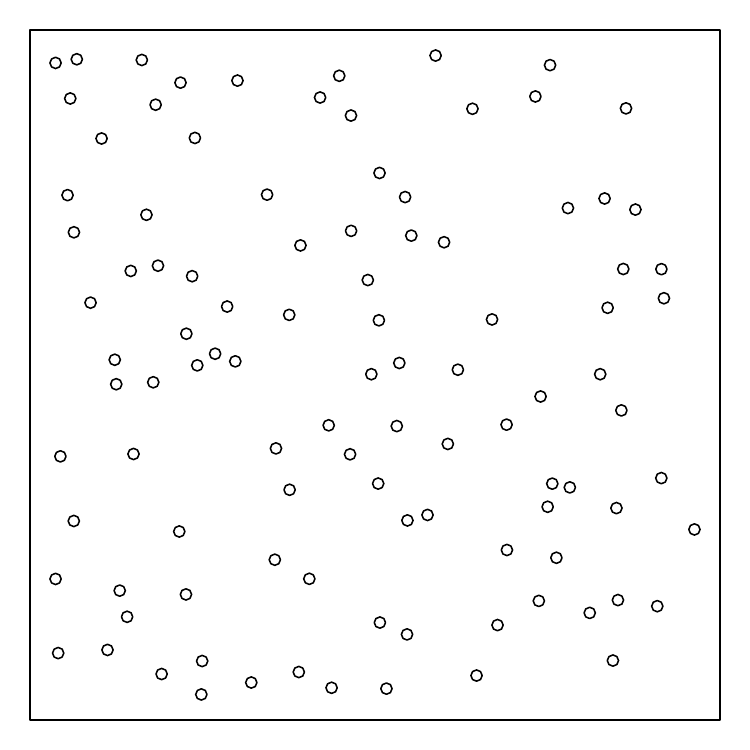}  \\
$\rho=100$, $\alpha=0.01$  & $\rho=100$, $\alpha=0.03$ & $\rho=100$, $\alpha=0.05$  \\
 % \includegraphics[scale=.3]{Xbessell1000a005_sample.pdf} & \includegraphics[scale=.3]{Xbessell1000a010_sample.pdf} & \includegraphics[scale=.3]{Xbessell1000a015_sample.pdf}  \\
 % $\rho=1000$, $\alpha=0.005$ &  $\rho=1000$, $\alpha=0.01$ & $\rho=1000$, $\alpha=0.015$  \\
\includegraphics[scale=.3]{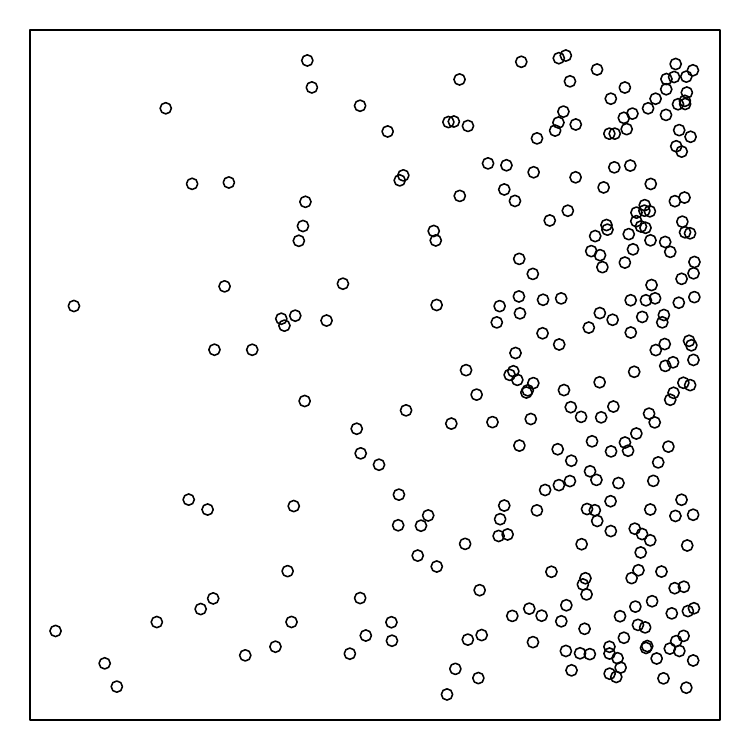} &  \includegraphics[scale=.3]{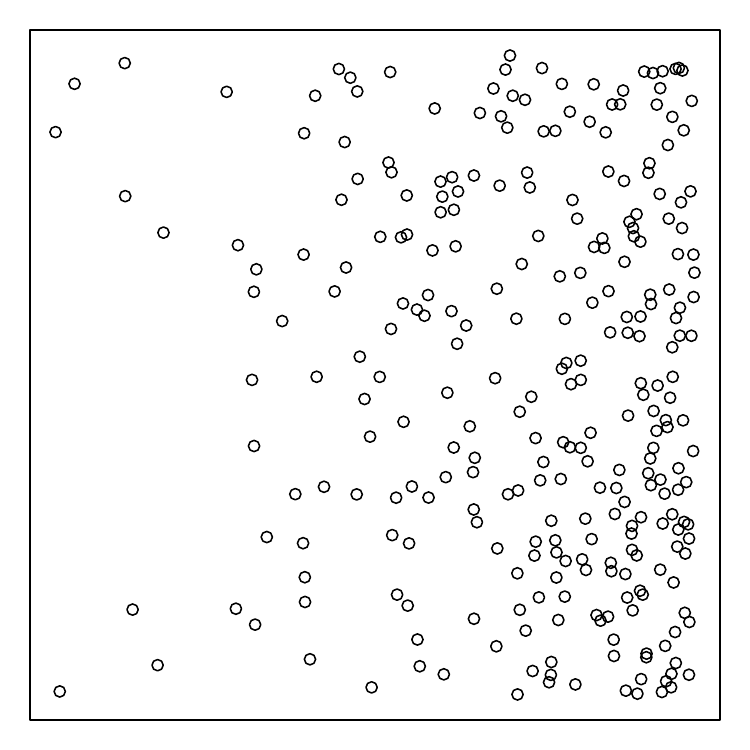} &  \includegraphics[scale=.3]{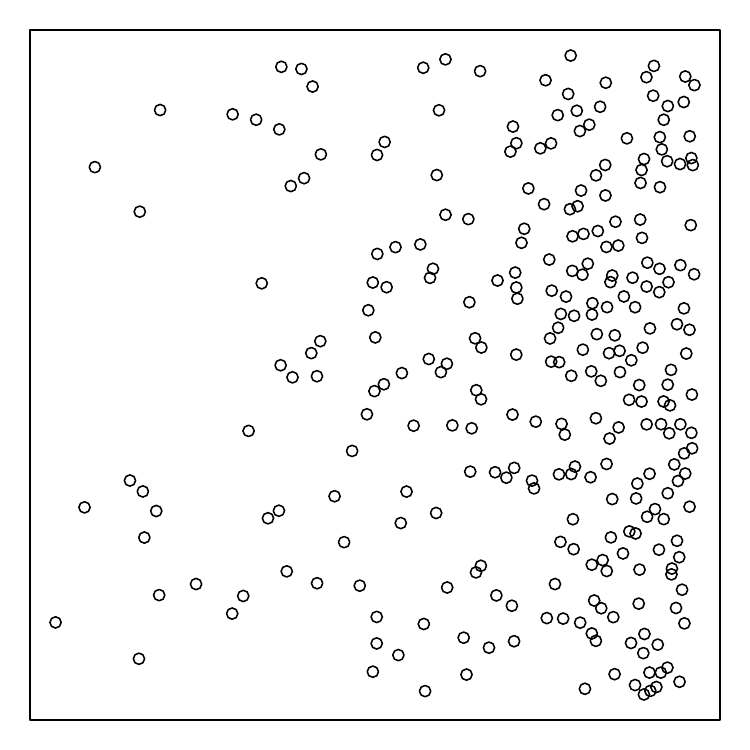} \\
$\rho=\rho(u)$, $\alpha=0.005$ & $\rho=\rho(u)$, $\alpha=0.01$ & $\rho=\rho(u)$, $\alpha=0.015$\\
\end{tabular}}
\caption{{\small Examples of point patterns simulated from a Bessel-type DPP on $[0,1]^2$ for different values of $\rho$ and $\alpha$. For the last row, $\rho(x,y)= 20\exp(4x)$.}}
\label{fig:samples} 
\end{center}
\end{figure}

\begin{table}
\centering
\resizebox*{!}{0.95\textwidth} {
\begin{tabular}{|rrcrrrr|r|}
 \hline
$\rho$ & $\alpha$ & &$R=0.05$ & $R=0.1$ & $R=0.25$ & Adaptive & $\hat R$ \\  \hline
50 & 0.02 & \sc{rmse:} & 5.84 (0.15)& 5.83 (0.17)& 6.29 (0.19)& 5.97 (0.18)& 0.047 \\
  && \sc{time:} & 0.43 & 0.48 & 0.68 & 0.64 & (0.020)\\  \vspace{-0.3cm}&&&&&&&\\ 
  & 0.04  & \sc{rmse:} & 15.60 (0.44)& 9.18 (0.20) & 9.19 (0.22)& 9.25 (0.21)& 0.106 \\ 
   && \sc{time:} &  0.48 & 0.50 & 0.68 & 0.73 &(0.037)\\
  \vspace{-0.3cm}&&&&&&&\\
& 0.07  & \sc{rmse:}  & 13.32 (0.33)&  8.25 (0.23)&  8.22 (0.24)&  8.15 (0.24)& 0.147 )\\
  && \sc{time:} &  0.50 & 0.45 & 0.59 & 0.72 &(0.050\\ \vspace{-0.3cm}&&&&&&&\\\hline
100 & 0.01   & \sc{rmse:} &  2.44 (0.08)& 2.45 (0.08)& 2.58 (0.09)& 2.63 (0.09) & 0.024 \\ 
 && \sc{time:} &  0.44 & 0.57 & 1.22 & 0.70 &(0.009)\\  \vspace{-0.3cm}&&&&&&& \\ 
 & 0.03  & \sc{rmse:}  & 5.34 (0.13)& 5.12 (0.13)& 5.28 (0.14)& 5.27 (0.13)& 0.064 \\  
 && \sc{time:} & 0.40 & 0.47 & 0.98 & 0.70  &(0.019)\\  \vspace{-0.3cm}&&&&&&&\\ 
 & 0.05  & \sc{rmse:}  & 5.78 (0.12)& 4.43 (0.12)& 4.50 (0.10)& 4.53 (0.12)& 0.139  )\\  
 && \sc{time:} &  0.52 & 0.56 & 1.16 & 0.95&(0.022 \\ \vspace{-0.3cm}&&&&&&&\\\hline
1000 & 0.005  & \sc{rmse:} &  0.67 (0.02)& 0.88 (0.02) & 0.83 (0.02)& 0.72 (0.02)& 0.015 \\ 
 && \sc{time:} &  3.83 & 19.04 & 110.07 & 9.38 & (0.003)\\ \vspace{-0.3cm}&&&&&&& \\
& 0.01   & \sc{rmse:} & 0.57 (0.01)& 0.59 (0.02)& 0.61 (0.01)& 0.56 (0.01)& 0.028 \\ 
 && \sc{time:} & 2.68 & 10.40 & 60.79 & 6.84 &(0.005)\\  \vspace{-0.3cm}&&&&&&&\\ 
 & 0.015 & \sc{rmse:}  &  0.47 (0.01)& 0.46 (0.01)& 0.52 (0.01)& 0.47 (0.01)& 0.026\\  
 && \sc{time:} & 2.53 & 9.81 & 55.78 & 7.75&  (0.002)\\  \vspace{-0.3cm}&&&&&&&\\\hline
Inhom & 0.005  & \sc{rmse:} &   1.58 (0.04)& 1.65 (0.04)& 1.66 (0.04)& 1.61 (0.04)& 0.014 \\ 
 && \sc{time:} &  0.89 &  2.50 & 10.30 & 1.19  & (0.005)\\ \vspace{-0.3cm}&&&&&&& \\
& 0.01   & \sc{rmse:} &  1.34 (0.03)& 1.36 (0.03)& 1.36 (0.03)& 1.32  (0.03)& 0.025 \\ 
 && \sc{time:} & 0.76 & 1.86 & 7.66 & 1.22 &(0.008)\\  \vspace{-0.3cm}&&&&&&&\\ 
 & 0.015 & \sc{rmse:}  &  1.43 (0.03)&1.47 (0.03)&1.48 (0.03)&1.40  (0.03)& 0.030 \\  
 && \sc{time:} &  0.86 &1.90& 7.46& 1.40 & (0.006)\\
  \hline
\end{tabular}}
\caption{{\small Estimated root mean square errors ($\times 10^3$) and mean computation time (in seconds) of $\hat\alpha$ for a Bessel-type DPP on $[0,1]^2$, for different values of $\rho$ and $\alpha$. The 3 first estimators use the test function \eqref{eq:CL_R} with $R=0.05$, $R=0.1$ and $R=0.25$ respectively, while the last estimator is the adaptive version based on \eqref{eq:f_adapt}. The standard errors of the RMSE estimations are given in parenthesis. The last column gives the averages of "practical ranges" (i.e. maximal solution to $|g(r)-1|=0.01$) used for the adaptive estimator, along with their standard deviations in parenthesis. For each value of $\rho$ and $\alpha$, these quantities are computed from 1000 simulations where all four estimation methods have converged.} } \label{tab:bessel} \end{table}

%\newpage

\section{Application}\label{sec:application}

To illustrate the practical importance of our adaptive estimating function and our asymptotic results we consider the problem of fitting a DPP model to the point pattern data in the left plot of Figure~\ref{fig:dataexample}. This dataset   collected by  \cite{numata64} records the locations of 204 seedlings and saplings of Japanese black pines in an observation window of dimension  10m by 10m. It has previously been analysed in \cite{ogata:tanemura:86} using an inhomogeneous Gibbs model and later in \cite{Lavancier15} using an inhomogeneous DPP with kernel of the form \eqref{dppSOIRS} with a cubic polynomial in the Cartesian coordinates for the log-intensity and $C(u)=\exp(-\|u\|^2/\alpha^2$). The estimation in \cite{Lavancier15} was carried out using a two-step procedure where  the intensity parameters were fitted in the first step by the Poisson likelihood method and in the second step $\alpha$ was estimated by minimisation of a contrast function based on the pair correlation function. This gave $\hat\alpha = 0.226$ and the fit was judged to be satisfying based on several goodness of fit envelope tests. However this second step relies on the arbitrary choice of several tuning parameters similar to $R$ and no confidence intervals were provided. We also fit the same inhomogeneous DPP model by the two-step approach (detailed in Section~\ref{sec:asympt_sep}) but using in the second step the test function \eqref{eq:CL_R} (for the non-adaptive approach) or \eqref{eq:f_adapt} (for our adaptive version).

\begin{figure}
\begin{center}
\includegraphics[height=6.3cm]{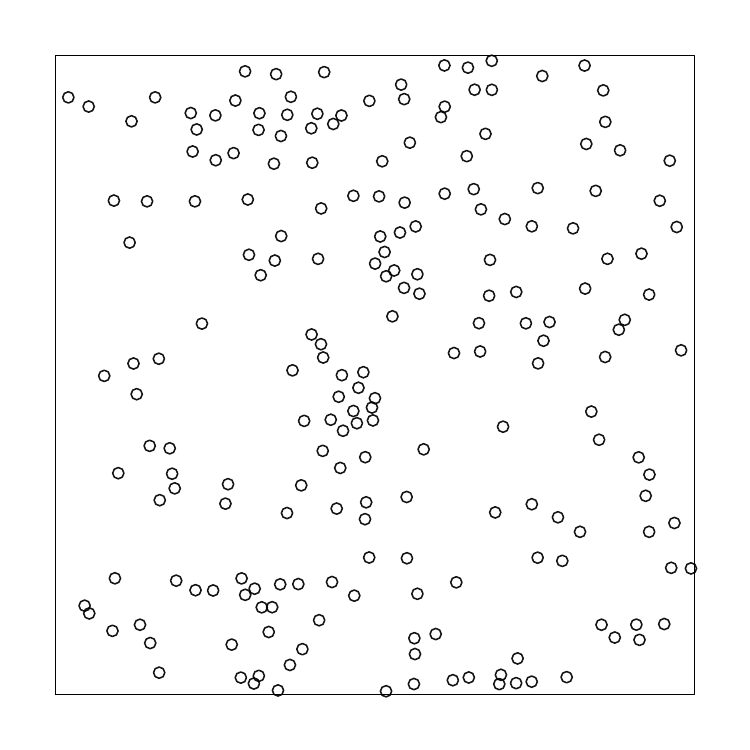} 
\includegraphics[height=6cm]{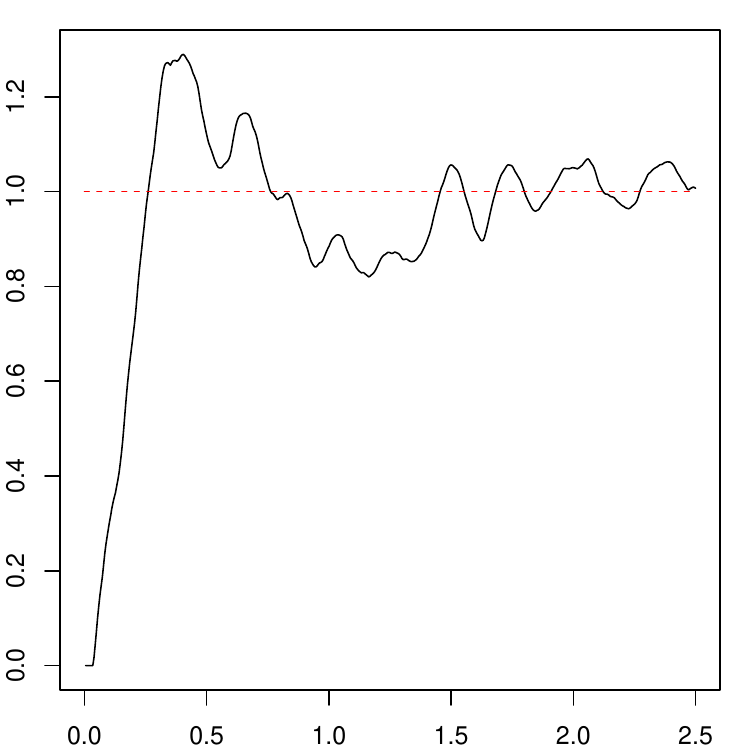} 
\caption{{\small Japanese pines dataset and  kernel estimate of its pair correlation function.}}
\label{fig:dataexample} 
\end{center}
\end{figure}

 For the non-adaptive approach, the default choice of $R$ provided by \texttt{spatstat} \citep{baddeley:rubak:turner:15} is one quarter of the smallest side length of the observation window, which in the current case gives $R=2.5$.
This choice of $R$ is subject to criticism since it does not at all take into account the correlation properties of the data generating point process. Alternatively one could, following the references mentioned in Section~\ref{sec:adaptiveR}, choose an $R$ based on inspection of the kernel estimate of the pair correlation function shown in the right plot in Figure~\ref{fig:dataexample}. This suggest using a value of $R$ around 0.4. However, this last approach completely eludes a theoretical underpinning. First, the asymptotic properties of the kernel estimator itself are complicated and second it is not possible to handle mathematically the visual assessment of $R$. 
We use instead our adaptive procedure and obtain $\hat\alpha=0.241$ for the range parameter and the adaptively chosen value of $\hat R=0.37$, which is in agreement with the visual inspection of the pair correlation function. 

Based on our Theorem~\ref{th:sep} in Section~\ref{sec:asympt_sep} we further in a standard way obtain a 95\% approximate confidence interval for $\alpha$ (estimate $\pm$ 1.96 times asymptotic standard deviation). For ease of implementation we use a parametric bootstrap to estimate the asymptotic standard deviation (alternatively one could use numerical integration to compute the asymptotic covariance matrix). More precisely, we generate $1000$ realisations of the fitted model and refit the model for each simulation. The empirical standard deviation of the resulting simulated estimates is then an estimate of the standard deviation of $\hat \alpha$. Note that our consistency result in Theorem~\ref{th:sep} is a requirement for the validity of this approach (see for instance \cite{Beran1997}). We obtain the specific estimate $0.026$ and hence the approximate 95\% confidence interval $[0.19;0.29]$ for $\alpha$. This result confirms that there is significant inhibition between the Japanese pines and in particular provides strong evidence against the inhomogeneous Poisson process model ($\alpha=0$). Note that a similar bootstrap approach is not possible in the non-adaptive case where $R$ is chosen by visual inspection, since the variability of this choice can not be included in an automatic procedure.

\section{Discussion}

In this paper we provide a very general asymptotic framework for estimating function inference for spatial point processses with known joint intensities. Specific asymptotic results are obtained for determinantal point processes.

 The performance of second order estimating functions depends strongly on a tuning parameter $R$ that controls which pairs of points are used in the estimation. Although not statistically optimal, our adaptive procedure for selecting this tuning parameter is intuitively appealing and easy to implement. The method depends on a new tuning parameter $\epsilon$ for which it is easier to identify reasonable values than for the original tuning parameter $R$. The resulting estimation procedure is computationally tractable and performs well in terms of mean squared error in the simulation studies considered. It moreover seamlessly integrates with the asymptotic results where the use of the adaptive method poses no extra theoretical difficulties. 

Though we focus in this paper on determinantal point processes, the adaptive method is applicable for any spatial point process with known pair correlation function. As an example we provide in Section~\ref{supp-sec:thomas} of the supplementary material a simulation study in case of a cluster process. 
 
\section*{Acknowledgements}
Rasmus Waagepetersen was supported by The Danish Council for Independent Research | Natural Sciences, grant DFF - 7014-00074 "Statistics for point processes in space and beyond", and by the "Centre for Stochastic Geometry and
Advanced Bioimaging", funded by grant 8721 from the Villum Foundation.

\bibliographystyle{royal}

\bibliography{masterbib}

\vspace{0.4cm}

\noindent Fr\'ed\'eric  Lavancier, Jean Leray Mathematics Institute, University of Nantes, 2 rue de la Houssini\`ere, 44322 Nantes Cedex 3,  France.\\
E-mail:~frederic.lavancier@univ-nantes.fr

\appendix

\section{Appendix}\label{sec:proof_general}

%\section{Assumptions for Theorem~\ref{th:general}}\label{sec:proof_general}

Our general Theorem~\ref{th:general} depends on a number of  assumptions. The setting is the same as in Section~\ref{sec:asympt_global}. We moreover define $\mbox{diam}(x)$ as the largest distance between two coordinates of $x$. 
The assumptions \eqref{Ass:reg} through \eqref{Ass:LimitMatrix} are mainly related to the test functions $f_i$, while for $X$ we assume \eqref{Ass:intenseBound} through \eqref{Ass:CLT}.

\begin{enumerate}[label=(F\arabic*),ref=F\arabic*]

\item For all $i=1,\dots,l$ and for all $x\in(\R^d)^{q_i}$, $\theta\mapsto f_i(x;\theta)$ is twice continuously differentiable in a neighbourhood of $\theta^*$. Moreover, the first and second derivative of $f_i$ with respect to $\theta$ are bounded with respect to $x \in (\R^d)^{q_i}$ uniformly in $\theta$ belonging to this neighbourhood.
\label{Ass:reg}

\item There exists a constant $R>0$ such that for all $\theta$ in a neighbourhood of $\theta^*$, all functions $x\mapsto f_i(x;\theta)$ vanish when $\mbox{diam}(x)>R$. \label{Ass:fbound}
\end{enumerate}
Define the matrices $H_n(\theta)$ by 
\begin{equation*}H_n(\theta)=\left (
\begin{array}{c}
    H_n^1(\theta) \\
    \vdots \\
    H_n^l(\theta) \\
\end{array}
\right ),\end{equation*}
where for all $i$ 
$$H_n^i(\theta) :=\frac{1}{|W_n|}\int_{W_n^{q_i}} f_i(x;\ta) \nabla_\theta\rho^{(q_i)}(x;\ta)^T \der x.$$ 

\begin{enumerate}[label=(F\arabic*),ref=F\arabic*]
\addtocounter{enumi}{2}
\item   The matrices $H_n(\theta^*)$ satisfy $$\liminf_{n\rightarrow\infty}\left(\inf_{\|\phi\|=1}\phi^T H_n(\theta^*)\phi\right)>0.$$
\label{Ass:LimitMatrix}

\item[(F3')] There exists a neighbourhood of $\theta^*$ such that for all $n$ high enough and all $\theta$ in this neighbourhood, $H_n(\theta)$ is invertible and $\|H_n(\theta)^{-1}\|$ is uniformly bounded with respect to $n$ and $\theta$, where $\| \cdot \|$ stands for any matrix norm.

\end{enumerate}

\begin{enumerate}[label=(X\arabic*),ref=X\arabic*]
\item For all $\theta$ in a neighbourhood of $\theta^*$ and all $q_i$, $i=1,\dots,l$, the intensity functions $x\mapsto \rho^{(q_i)}(x;\theta)$ are well-defined and bounded.
Moreover, $\theta\mapsto \rho^{(q_i)}(x;\theta)$ is twice continuously differentiable in a neighbourhood of $\theta^*$, for all $x\in(\R^d)^{q_i}$. Finally, the first and second derivative of $\rho^{(q_i)}$ with respect to $\theta$ are bounded with respect to $x\in(\R^d)^{q_i}$ uniformly in $\theta$  belonging to this neighbourhood.
\label{Ass:intenseBound}

\item For all $q_i$, $i=1,\dots,l$, the intensity functions $\rho^{(q_i)}(\cdot;\ta^*),\cdots,\rho^{(2q_i)}(\cdot;\ta^*)$ of $X$ are well-defined. Moreover, the intensity functions $\rho^{(q_i)}(\cdot;\ta^*),\cdots,$ $\rho^{(2q_i-1)}(\cdot;\ta^*)$ are bounded  and for all bounded sets $W\subset\R^d$ there exists a constant $C_0>0$, so that $ \int_W \varphi_i(x_1) \der x_1 < C_0$, $i=1,\ldots,l$ where $\varphi_i$ is the function
\begin{multline*}
\varphi_i :x_1\mapsto \sup_{\diam(x)<R}\sup_{\diam(y)<R}\sup_{y_1\in W} \rho^{(2q_i)}(x_1,x_2,\cdots,x_{q_i},y_1,\cdots,y_{q_i};\theta^*)\\
-\rho^{(q_i)}(x_1,x_2,\cdots,x_{q_i};\theta^*)\rho^{(q_i)}(y_1,\cdots,y_{q_i};\theta^*)
\end{multline*}
with $R$ coming from \eqref{Ass:fbound}.
\label{Ass:intense}

\item $X$ satisfies the central limit theorem
\begin{equation*} \Sigma^{-1/2}_n  e_n(\theta^*) \cvlaw\mathcal{N}(0, I_p),\end{equation*}
where $e_n$ is defined in Section~\ref{sec:asympt_global} and $\Sigma_n=\var(e_n(\theta^*))$.
\label{Ass:CLT}
\end{enumerate}

Assumptions \eqref{Ass:reg} and \eqref{Ass:fbound} are basic regularity conditions on the $f_i$'s. Similarly \eqref{Ass:intenseBound} and \eqref{Ass:intense} ensure that the intensity functions of $X$ exist and are sufficiently regular. 
The technical assumptions are in fact  \eqref{Ass:LimitMatrix}  (or  (\ref{Ass:LimitMatrix}')) and \eqref{Ass:CLT}.  While the latter strongly depends on the underlying point process (see \cite{waagepetersen:guan:09} for Cox processes and \cite{poinas17} for DPPs), the former 
can be simplified in some cases. For example, if $H_n(\theta^*)$ are
symmetrical matrices for all $n$ then \eqref{Ass:LimitMatrix} writes
$\liminf_n \lambda_{\min}(H_n(\theta^*))>0$ where
$\lambda_{\min}(H_n(\theta^*))$ denotes the smallest eigenvalue of
$H_n(\theta^*)$. If the matrices $H_n(\theta^*)$ are not symmetrical,
Assumption (\ref{Ass:LimitMatrix}') will be preferred since
\eqref{Ass:LimitMatrix} does not translate well for non-symmetrical
matrices.  Furthermore, if $X$ is stationary, all $f_i$'s are
invariant by translation, and the sequence of windows
  $\{W_n\}_{n \ge 1}$ satisfies \eqref{Ass:ShapeWindow} in
  Section~\ref{sec:asymptDPP_global}, then $H_n(\theta)$ converges
towards a matrix $H(\theta)$ explicitly given in Lemma
\ref{lem_general} below. Assumption \eqref{Ass:LimitMatrix} thus simply becomes $\inf_{\|\phi\|=1}\phi^T H(\theta^*)\phi>0$ and (\ref{Ass:LimitMatrix}') is satisfied whenever $H(\theta^*)$ is invertible by continuity of $H(\theta)$. In specific applications of Theorem~\ref{th:general}, further conditions on the sequence of observation windows $\{W_n\}_{n \ge 1}$ may be required, see e.g.\ \eqref{Ass:ShapeWindow} in Section~\ref{sec:asymptDPP_global}.
\begin{lem}\label{lem_general}
Assume \eqref{Ass:ShapeWindow}, \eqref{Ass:intenseBound}, \eqref{Ass:fbound} and let $\theta\in\R^p$. Suppose that all $\rho^{(q_i)}(\cdot;\theta)$'s and $f_i( \cdot;\theta)$'s are invariant by translation,  i.e. $f_i(u_1,u;\theta)=f_i(0,u-u_1;\theta)$ where $u$ is the vector $(u_2,\cdots,u_{q_i})$ and $u-u_1=(u_2-u_1,\cdots,u_{q_i}-u_1)$.
 If $u\mapsto f_i(0,u;\theta)$ is integrable for all $i$ such that $q_i\geq 2$, then $H_n(\theta)$ converges to a matrix $H(\theta)$. In particular, for all $i$ we have
$$\lim_{n\rightarrow\infty}H^i_n(\theta)=\int_{\|t\|\leq R} f_i(0,t;\ta) \nabla_\theta\rho^{(q_i)}(0,t;\ta)^T  \dd t.$$
\end{lem}
The proof of this lemma is available in the supplementary material.

\end{document}